\documentclass[12pt]{article}
\usepackage{latexsym}
\usepackage{amsfonts}
\usepackage{amsmath}
\usepackage{amsmath,amssymb,proof1,latexsym}
\usepackage[all]{xypic}

\makeindex

\newcommand{\imp}{\!\rightarrow\!}

\newcommand{\mpn}{\medskip\par\noindent}

\newcommand{\pa}{{\sf PA}}

\newcommand{\proves}{\vdash}

\newcommand{\gn}[1]{\ulcorner {#1} \urcorner }

\newcommand{\lc}[1]{#1\!\!:\!\!}

\newtheorem{Prop}{\bf Proposition}

\newtheorem{Theor}{\bf Theorem}

\newtheorem{Lemma}{\bf Lemma}

\newtheorem{Coro}{\bf Corollary}

\newtheorem{Fact}{\bf Fact.}

\newtheorem{Remark}{\bf Remark}

\newtheorem{Claim}[enumi]{Claim}

\newtheorem{defin}{\bf Definition}

\newtheorem{exam}{\bf Example}

\newtheorem{notat}{\bf Notation.}

\newenvironment{proof}{{\bf Proof.}}{\hfill $\slot$}
\newcommand{\slot}{\hfill \mbox{$\Box$}\vspace{\parskip}\\}
\newtheorem{Comment}{\bf Comment}

\begin{document}

\title{The Provability of Consistency}

\author{Sergei Artemov\\ \\
 {\small The Graduate Center, the City University of New York}\\
{\small  365 Fifth Avenue, New York City, NY 10016}\\
{\small {\tt sartemov@gc.cuny.edu}} }
\date{June 19, 2020}
\maketitle

\begin{abstract}
We offer a mathematical proof of consistency for Peano Arithmetic \pa\ formalizable in \pa.
This result is compatible with G\"odel's Second Incompleteness Theorem since our consistency proof does not rely on the representation of consistency as a specific arithmetical formula. 
Our findings show that formal theories can finitely formalize proofs of certain properties presented as schemes without reducing the presentation of those properties to a single formula. We outline a theory of proving schemes in \pa. 

\end{abstract}

\section{Introduction}
In the 1920s, Hilbert announced a program of establishing consistency of formal mathematical theories by trusted/finitary means (cf., e.g., \cite{Zach16}).  
It is widely believed that G\"odel's Second Incompleteness Theorem undermined Hilbert's program. Here is a typical expression of the impossibility reading of G\"odel's Theorem:
\[ \mbox{\it ``there exists no consistency proof of a system that can be formalized in the system itself"} \]
(Encyclop{\ae}dia Britannica, \cite{EB20}).

For first-order Peano Arithmetic \pa, G\"odel's Second Incompleteness Theorem claims that some arithmetical formula 
\[{\sf Con}_{\sf PA},\]
which {\em can} be read as an internalized consistency assertion, is not derivable in \pa. Together with the widely accepted {\it Formalization Principle}: 
\begin{quote}
{\it any contentual reasoning within the postulates of \pa\
can be internalized as a formal derivation in \pa,}
\end{quote}
and the tacit additional assumption, call it the \textit{``${\sf Con}_{\sf PA}$ as Consistency Principle},"
\begin{quote}
{\it any contentual proof of consistency of \pa\ within the postulates of \pa\
can be internalized as a formal \pa-derivation of the formula ${\sf Con}_{\sf PA},$}
\end{quote}
G\"odel's Second Incompleteness Theorem then would imply that \pa-consistency cannot be established by means of \pa. Likewise, no consistent extension of \pa\ proves its own consistency. 

Yet neither Hilbert nor G\"odel accepted these conclusions.
Hilbert (\cite{HB34}) rejected the impossibility reading of G\"odel's Second Incompleteness Theorem in the context of proving consistency:
\begin{quote}
\textit{``the view ...  that certain recent results of G\"odel show that
my proof theory can't be carried out, has been shown to be erroneous. In fact that result
shows only that one must exploit the finitary standpoint in a sharper way for the farther
reaching consistency proofs."}
\end{quote}
G\"odel directly challenges the Formalization Principle. In \cite{God31}, G\"odel writes:
\begin{quote} {\it 
``it is conceivable that there exist finitary proofs that cannot be expressed in the formalism of [our basic system]."}
\end{quote}

There are definite indications that the late G\"odel remained skeptical about the impossibility reading of his Incompleteness Theorem. 
Gerald Sacks (\cite{Sac07}) recalled G\"{o}del claiming around 1961--1962 that some iteration of Hilbert's consistency program remained feasible. 
G\"{o}del {\it ``did not  think''} the objectives of Hilbert's program 
{\it ``were erased''} by the Incompleteness Theorem, and G\"{o}del believed (according to G.~Sacks) 
 it left Hilbert's program 
 \[ \mbox{\it ``very much alive and even more interesting than it initially was.''} \footnote{Thanks to Dan Willard (\cite{Wil20}) for pointing at this Sacks/G\"odel quotes.} \]


Now we delineate a class of contentual reasoning formalizable in \pa\ that is not excluded by G\"odel's Second Incompleteness Theorem. The above, plus the observation that our consistency proof is a member of this class, renders the impossibility reading of G\"odel's Second Incompleteness Theorem unwarranted.

Consider informal elementary number theory containing recursive identities for $+$, and $\times$ as well as the induction principle; we will call this theory \textit{informal arithmetic}.  Formal arithmetic \pa\ is the conventional formalization of informal arithmetic. Any formal statement of \pa\ can be read as a statement in informal arithmetic. Any proof of such a statement in informal arithmetic in this paper can be naturally formalized as a \pa-derivation of this statement. 

Consider the property of \textit{Complete Induction}, $\cal CI$, 
\[ \mbox{\it for any given formula $\psi$, if for all $x$ $\forall y<x\ \psi(y)$ implies $\psi(x)$, then $\ \forall x\psi(x)$},\]
and the textbook proof of it in informal arithmetic: take an arbitrary $\psi$ and apply the usual induction to $\varphi(x)=$ $\forall y< x\ \psi(y)$ to get the $\cal CI$ statement  ${\cal CI}(\psi)$ for $\psi$.
\medskip\par
This is a correct mathematical proof formalizable in \pa\ using the standard machinery of G\"odel numbering. Indeed, we build a simple \textit{selector} function and a \pa-proof that for any $\psi$, the selector returns a code of a proof of ${\cal CI}(\psi)$.

While the above mathematical proof of the \textit{Complete Induction} property is clearly a proof done by means of informal arithmetic, \textit{Complete Induction} itself cannot be represented by a single formula of \pa, because \pa\ is not finitely axiomatizable. 

This basic example shows that in (un)provability analyses of contentual properties, it is imperative to look beyond single arithmetical formulas for representations of a given property. If we do not allow for this, some classes of formalizable contentual proofs will be excluded from consideration, and some mathematically provable fundamental properties will be left without formal proofs. 

We show that \textit{Consistency} is one such property. 
By G\"odel's Second Incompleteness Theorem, \pa, since consistent, does not prove the specific arithmetical formula ${\sf Con}_{\sf PA}.$
However, consistency of \pa\ in its original formulation, as a property of finite sequences of formulas, can be stated as follows, 
\begin{equation}\label{Hconsistency}
\mbox{\it for any finite sequence $S$ of formulas, $S$ is not a \pa-derivation of $\bot$,}\footnote{$\bot$ is a propositional constant for contradiction which in arithmetical context can be interpreted as the formula $(0\! = \! 1)$.}
\end{equation}
or, equivalently, 
\[ \mbox{\it a given \pa-derivation $S$ does not contain $\bot$}, \]
for its proof does not require internalization as a single arithmetical formula. 

In this paper, we show that \textit{Consistency} admits a direct proof in informal arithmetic and that this proof is formalizable in \pa. For a given \pa-derivaton $S$, we find an arithmetically definable invariant 
\[ {\cal I}_S\]
and establish in informal arithmetic that for each $\varphi$ in $S$, ${\cal I}_S(\varphi)$ holds, ${\cal I}_S(\bot)$ does not hold, hence $\bot$ does not occur in $S$.

Proving the consistency property of \pa\ directly, without \textit{a priori} internalization, avoids limitations imposed by G\"odel's Second Incompleteness Theorem. This result suggests reconsidering the aforementioned popular belief that the consistency of \pa\ cannot be established by means formalizable in \pa. 

As a generalization, we suggest to follow mathematical practices and admit that formalizable proofs of the aforementioned kind, which we call \textit{proofs of schemes}, are admissible formal ways of proving properties in a theory.

\section{Groundwork}

Peano Arithmetic, \pa, is a formal first-order theory containing constant $0$, functions $^\prime$ (successor), $+$, $\times$, and the usual recursive identities for these functions. \textit{Numerals} are terms   
\[ 0, 0^\prime, 0^{\prime\prime}, 0^{\prime\prime\prime}, \ldots \]
representing natural numbers. 
In addition, \pa\ has the standard {\it induction principle}: for each formula $\varphi(x)$, the formula  $\mbox{\it Ind}_\varphi$ is postulated: 
\[ [\varphi(0)\wedge \forall x(\varphi(x)\imp\varphi(x^\prime))] \imp \forall x \varphi(x).
\]
The induction principle \textit{``for each $\varphi$,  $\mbox{\it Ind}_\varphi$"} cannot be represented in \pa\ by a single sentence. Induction in \pa\ is a \textit{scheme} of \pa-formulas $\mbox{\it Ind}_\varphi$ with the parameter $\varphi$. 

Every primitive recursive function is represented in \pa\ by a corresponding (definable) term: for notational convenience we can assume that terms for all primitive recursive functions are already present in the language of \pa\ along with the defining recursive conditions. So, if for a primitive recursive function $f$, $$f(n)=m,$$ then \pa\ proves this fact 
\[ \pa\proves f(n)=m. \] 
Consequently, any primitive recursive relation $R(x_1,\ldots,x_k)$ is naturally represented in \pa\ as well, and $$R(n_1,\ldots,n_k)$$ yields 
\[ \pa\proves R(n_1,\ldots, n_k) .\]

Let $$u:v$$ be the standard (primitive recursive) proof predicate in \pa\ stating 
\[ \mbox{\it $u$ is a code of a \pa-proof of a formula having code $v$.} \] 
In particular, $p$ is a \pa-proof of $\varphi$ iff $\gn{p}:\gn{\varphi}$ holds with $\gn{X}$ denoting the G\"odel number of $X$. We omit notation ``$\gn{\ }$" when safe. Within these conventions, 
\[ \mbox{\it $p$ is a \pa-proof of $\varphi$}\ \ \Leftrightarrow\ \ p:\varphi .\]

\subsection{Foundational view}

Though \pa\ is generally accepted to be consistent, there has been no foundationally clean explanation of \textit{why} \pa\ is consistent. Existing proofs in other theories of the consistency of \pa, some of them intuitively convincing, as well as soundness arguments, rely on yet stronger assumptions, and thus do not address the foundational issue in its strict sense. 

The \pa-consistency definition (\ref{Hconsistency}) is a contentual statement about syntactic objects - formal derivations. 
The traditional approach:
\[  \mbox{\it internalize {\em (\ref{Hconsistency})} as a \pa-sentence and analyze provability of that sentence in \pa},\]  cannot yield consistency of \pa, since an inconsistent theory vacuously proves anything. So, a proof of \pa-consistency should be contentual, which makes an \textit{a priori} formalization of (\ref{Hconsistency}) useless for proving consistency.

A foundationally meaningful proof of the consistency of \pa\ by means of \pa\ would be one that proves the consistency of \pa\ by means of informal arithmetic which includes usual finite combinatorics naturally represented by computable functions/terms. We offer such a proof together with its  complete \textit{a posteriori} formalization in \pa, basically to verify that this proof does not use assumptions from outside \pa. 

From a different perspective, one can assume Mathematics and Logic within the standard university curriculum: proofs, models, soundness, completeness, \textit{etc.} Within this framework, \pa\ is obviously consistent since it has a (standard) model. The question of 
\begin{equation}\label{conPAinPA}
\mbox{\it whether the consistency of \pa\ can be proved by means of informal arithmetic}
\end{equation}
becomes a mathematical problem. 
If nothing else, this is a typical problem of what can be done with limited tools, akin to \textit{doubling the cube} using only a compass and straightedge.
We show that (\ref{conPAinPA}) has an affirmative solution, which also answers the aforementioned foundational question.

\subsection{To internalize or not to internalize consistency \textit{a priori}}
We now discuss the principal bifurcation point. To our knowledge, this appears to have been overlooked in the near-century since the publication of G\"odel's result.
\medskip\par
Traditionally,  the mathematical statement of consistency (\ref{Hconsistency}) is represented in \pa\ by formula 
$\mbox{\sf Con}_{\sf PA}$:
\begin{equation}\label{cformula}
\forall x\ \neg\lc{x}\bot.
\end{equation}
This is a paradigm example of \textit{a priori formalization}. 

By G\"odel's Second Incompleteness Theorem, $\mbox{\sf Con}_{\sf PA}$ is not provable in \pa. Several other formulas naturally representing consistency did, however, turn out to be provable in \pa\ (Rosser, Feferman, Montagna, 
\textit{etc.} \cite{Fef60,Mon78,Ros36,Vis89}). 
It took contentual mathematical judgements to disqualify them in favor of $\mbox{\sf Con}_{\sf PA}$. 

To speak about  provability of ${\sf Con}_{\sf PA}$ in \pa\ semantically, we have to consider formula~(\ref{cformula}) 
in all models of \pa, including nonstandard models. In a given nonstandard model, the quantifier  \textit{``for all $x$"} spills over to nonstandard/infinite numbers, and hence does not match the quantifier \textit{``for any sequence $S$"} from the articulation of Hilbert's consistency in (\ref{Hconsistency}). This is because a finite sequence $S$ of formulas can only have a standard integer code. 
For this reason, G\"odel's Second Incompleteness Theorem cannot have to do with real \pa-derivations, which are all finite.

Moreover, it is well-known that consistency statements 
\[ \neg\lc{0}\bot, \neg\lc{1}\bot, \neg\lc{2}\bot, \neg\lc{3}\bot, \ldots \] hold in all models of \pa, hence only nonstandard/infinite ``proofs" of $\bot$ are possible in \pa-models. This observation demonstrates how internalization distorts the intrinsic nature of consistency and makes it unprovable for a non-essential reason: the language of \pa\ is too weak to sort out fake proof codes. These deficiencies of the language of \pa\  raise questions about using ${\sf Con}_{\sf PA}$ as a fair representation of Hilbert's consistency of \pa\ (\ref{Hconsistency}) and invite logicians to look for alternatives.

We suggest an alternative, as yet unexplored, route: proving consistency directly as a property of finite sequences of formulas, and then formalizing this proof in \pa\ to verify that no principles outside \pa\ have been used in this proof.

\section{Proof of \pa-consistency in informal arithmetic}\label{Proof}

Our proof of consistency for \pa\ proceeds with the following steps:

\begin{enumerate}
\item 
We first present a proof of consistency in its original combinatorial form in informal arithmetic. 
\item 
We then put forth a standard G\"odel  numbers formalization of  proof (1) in \pa. 
\end{enumerate}

\subsection{Partial truth definitions in \pa}

In the metamathematics of first-order arithmetic, there are well-known constructions called \textit{partial truth definitions}, cf. \cite{Buss98, HP17, Pud98, Smo85}. Namely, for each $n=0,1,2,\ldots$ we build, in a primitive recursive way, a $\Sigma_{n+1}$-formula 
\[ \mbox{\it Tr}_n(x,y), \]
called the \textit{truth definition for $\Sigma_{n}$-formulas}, which satisfies natural properties of a truth predicate formulated in Proposition~\ref{p1}.
Intuitively, when $\varphi$ is a $\Sigma_{n}$-formula and $y$ is a sequence encoding values of the parameters in $\varphi$, then $\mbox{\it Tr}_n(\gn{\!\varphi\!},y)$ outputs the truth value of $\varphi$ on $y$. 

\begin{Prop}\label{p1}
{\em (cf. \cite{Buss98, HP17, Pud98, Smo85})}
The following three claims hold: 
\begin{enumerate}
\item $\mbox{\it Tr}_n(\gn{\!\varphi\!},y)$ satisfies the usual properties of truth with respect to Boolean connectives, quantifiers, and \textit{Modus Ponens} for each $\varphi\in\Sigma_n$, and these properties are derivable in \pa.
\item For any $\Sigma_{n}$-formula $\varphi$, \pa\ naturally proves Tarksi's condition:
\[ \mbox{\it Tr}_n(\gn{\!\varphi\!},y)\ \Leftrightarrow\ \varphi(y).\] 
In particular,  $\neg\mbox{\it Tr}_n(\gn{\bot},y)$ is naturally provable, i.e., \pa\ proves that the formula $\bot$ does not satisfy  $\mbox{\it Tr}_n$.
\item For any axiom $A$ of \pa\ from $\Sigma_{n}$, $\mbox{\it Tr}_n(\gn{\!A\!},y)$ is provable.
\end{enumerate}

\end{Prop}
Note that all proofs used in the demonstration of Proposition~\ref{p1} can be regarded as rigorous contentual arguments within informal arithmetic. The formal language of \pa\ is used here just for bookkeeping.

\subsection{A proof of \pa-consistency in informal arithmetic}\label{proof}

{\bf Phase 1}: We prove \pa-consistency in its original combinatorial form 
\[ \mbox{\it any given \pa-derivation $S$ does not contain $\bot$.}\]
\begin{itemize}
\item Given a finite \pa-derivation $S$, we first calculate $n$ such that all formulas from $S$ are from $\Sigma_{n}$. 
\item Then, by induction up to the length of $S$, we check that for any formula $\varphi$ in $S$ with parameters $y$, the property $\mbox{\it Tr}_n(\gn{\!\varphi\!},y)$ holds. This is an immediate corollary of Proposition \ref{p1}, since each axiom from $S$ satisfies $\mbox{\it Tr}_n$, and each rule of inference respects $\mbox{\it Tr}_n$. So, $\mbox{\it Tr}_n$ serves as an invariant for formulas from $S$. 
\item By Proposition \ref{p1}, $\bot$ does not satisfy $\mbox{\it Tr}_n$, hence is not in $S$. 
\end{itemize}
\hfill $\Box$
\mpn
This is a rigorous contentual proof of the consistency of \pa. Constructions  and required properties used in this argument are formalizable in \pa: partial truth definitions, compliance of truth definitions with \pa-derivation rules, {\it etc}. 
\mpn
{\bf Phase 2}:  We formalize phase 1.

Here is a description of a primitive recursive function $p(x)$, called \textit{selector}, connecting a numeral $d$ with the \pa-proof $p(d)$ of $\neg\lc{d}\bot$. All quantifiers used in the description of the procedure are bounded by given primitive recursive functions of $d$, the G\"odel number of a \pa-derivation $S$.

Given $d$ we first calculate $n=n(d)$ such that all formulas from $S$ are from $\Sigma_{n}$. 
For any formula $\varphi$ in $S$, starting with axioms, by induction up to the length of $S$, we build a \pa-proof of $\mbox{\it Tr}_{n}(\gn{\!\varphi\!},y)$. 
Since, by Proposition 1, \pa\ proves $\neg\mbox{\it Tr}_{n}(\gn{\bot},y)$, we construct a proof that $\bot$ is not in $S$. 

By this description, $p(x)$ is primitive recursive, and \pa\ proves 
\begin{equation}\label{ConS}
\forall x[\lc{p(x)}\neg \lc{x}\bot].
\end{equation}
Once again, the \pa-proof of (\ref{ConS}) is a formal certification that a given earlier contentual proof of \pa-consistency is formalizable in \pa. 


In this proof, $\mbox{\it Tr}_n$ serves as an invariant for formulas from $S$. G\"odel's Second Incompleteness Theorem does not rule out the possibility of having an invariant ${\cal I}_S$ for each $S$, but prohibits having such an invariant ${\cal I}$ uniformly for all derivations $S$. 

Indeed, suppose there is an arithmetical formula ${\cal I}(x,y)$ such that \pa\ proves 
\[  \forall x,y[\lc{x}y \imp  {\cal I}(x,y)] \]
and 
\[ \forall x\ \neg {\cal I}(x,\bot).\]
Then \pa\ proves 
\[ \forall x\ \neg \lc{x}\bot , \]
which is the internalized consistency formula ${\sf Con}_{\sf PA}$, and this is impossible by G\"odel's Second Incompleteness Theorem. 



\subsection{How far we can go}
As was shown in Section~\ref{proof}, informal arithmetic proves the consistency of \pa. Using this method, for which $T$ that extend \pa\ can we prove the consistency of $T$ in informal arithmetic? For example, can informal arithmetic prove the consistency of $\pa+ {\sf Con}_{\sf PA}$?  The answer is ``no."

Let  $\lc{x}_{\mbox{\it\tiny T}}\ \!\varphi$ be a shorthand of the natural arithmetical formula for the primitive recursive proof predicate in a theory $T$\footnote{We drop the subscript when $T=\pa$.}:
\[ \mbox{\it ``$x$ is a code of a proof of formula $\varphi$ in $T$," }\]
and $\Box_{\mbox{\it\tiny T}}\ \!\varphi$ denote {\it ``$\varphi$ is provable in $T$,"} i.e., $\exists x (\lc{x}_{\mbox{\it\tiny T}}\ \!\varphi)$. Then  ${\sf Con}_{\mbox{\it\tiny T}}$ is $\neg\Box_{\mbox{\it\tiny T}}\bot$. 

As a byproduct of our consistency proof, we can derive in \pa\ the {\it constructive consistency formula}, ${\sf CCon}_{\mbox{\it\tiny T}}$ (for $T=\pa$):
\[ \forall x\ \Box_{\mbox{\it\tiny T}} \neg \lc{x}_{\mbox{\it\tiny T}}\bot .
\] 
Studying provability of ${\sf CCon}_{\mbox{\it\tiny T}}$ turned out to be useful in the analysis of consistency proofs. 

Within the current approach to proving \pa-consistency we prove ${\sf CCon}_{\sf PA}$ too. Indeed, in \pa, from (\ref{ConS}), i.e., $\forall x\ \lc{p(x)}\neg \lc{x}\bot$, we immediately get the desired
\[ \forall x\exists y\ \lc{y} \neg \lc{x}\bot .
\] 

As was noticed by Sinclaire \cite{Sin19} and Kurahashi~\cite{Kur19}, independenly,  
\[ \mbox{\it \pa\ does not prove ${\sf CCon}_{\mbox{\it\tiny T}}$ for any $T\supseteq \pa+ {\sf Con}_{\sf PA}$}. \]
Indeed, suppose \pa\ proves ${\sf CCon}_{\mbox{\it\tiny T}}$ for $T = \pa+ {\sf Con}_{\sf PA}$. Then, equivalently,
\begin{equation}\label{?!?}
\pa\proves\forall x\Box\neg\lc{x}\Box\bot. 
\end{equation}
This does not hold, since otherwise $\pa\proves\Box\Box\bot\imp\Box\bot$, (which is not the case). Indeed, reason in \pa\ and assume $\Box\Box\bot$, i.e., $\exists x(\lc{x}\Box\bot)$. 
By a strong form of provable $\Sigma_1$-completeness,
$$\pa\proves \lc{x}\Box\bot \imp\Box (\lc{x}\Box\bot),$$
and we would have $\exists x\Box \lc{x}\Box\bot$. From (\ref{?!?}), we get $\Box\bot$. 

This could be loosely interpreted as stating that though informal arithmetic proves the consistency of \pa, it cannot prove consistency of $\pa + {\sf Con}_{\sf PA}$ by the given method.

\section{Proving properties in a general setting: schemes}\label{shemy}

We now present the previous argument in general proof-theoretical terms. A retroactive analysis shows that the presented \pa-consistency proof can be viewed as a proof in \pa\ of a corresponding arithmetical scheme. 

\begin{defin}
Let $\varphi$ be an arithmetical formula with a designated variable, or a vector of variables, $x$. A \textit{scheme} determined by $\varphi$ is a syntactic figure $\{\varphi\}$. A {\bf proof of a scheme 
$\{\varphi\}$} in \pa\ is a pair $\langle t, p\rangle$ where 
\begin{itemize}
\item $t$ is a primitive recursive term (called \textit{selector}),
\item $p$ is a \pa-proof of  $\ \forall x [t(x)\!:\!\varphi(x)]$. 
\end{itemize} 
\end{defin}
A similar approach works for schemes with non-numeral parameters. 
Let $S$ be a scheme of arithmetical formulas with parameter $\psi$ (think of \textit{Complete Induction}) and $s(x)$ a natural arithmetical term for the primitive recursive function that given a code of $\psi$ computes the code of $S(\psi)$. 
{\bf A proof of $S$ in \pa} is a pair $\langle t, p\rangle$ where $t$ is a primitive recursive term (\textit{selector}) and 
$p$ is a \pa-proof of  $\forall x [t(x)\!:\! s(x)]$.

\begin{defin}
A scheme $\{\varphi\}$ is 
\begin{itemize}
\item {\bf provable in \pa} if it has a proof in \pa, 
\item {\bf strongly provable in \pa} if $\pa\proves\forall x \varphi(x)$, 
\item {\bf weakly provable in \pa} if $\pa\proves \varphi(n)$, for each $n=0,1,2, \ldots$. 
\end{itemize}
\end{defin}


The principal idea of proving a scheme $\{\varphi\}$ is to represent in a finite way the contentual mathematical reasoning that certifies the collection of individual statements 
\begin{equation}\label{inf}
 \varphi(0),\ \varphi(1),\ \varphi(2),\ \ldots \ .
\end{equation}

As has been noted, reducing this problem to that of proving the internalized version of (\ref{inf}),   $\forall x \varphi(x)$, smuggles in nonstandard (infinite) numbers. For certain delicate properties, like \textit{Consistency}, this distorts the problem in an unacceptable way. There is no difference between (\ref{inf}) and $\forall x\varphi(x)$ semantically in the standard model of \pa, but proof methods are profoundly sensitive to their difference. 

\begin{Prop} 
$\ $ 

a) ``Strongly provable" {\em yields} ``Provable."

b) ``Provable" {\em yields} ``Weakly provable."
\end{Prop}
\begin{proof}
a) Suppose $q$ is a proof of $\forall x \varphi(x)$, i.e., 
\[  \lc{q}\forall x \varphi(x)).\]
By an easy transformation of proofs, find a primitive recursive term $t$ such that \pa\ proves 
\[ \forall x[\lc{t(x)}\varphi(x)].\]
Let $p$ be a \pa-proof of the latter, 
\[ \lc{p}\forall x[\lc{t(x)}\varphi(x)].\]
\par
b) Reason in informal arithmetic. From $\lc{p}\forall x[\lc{t(x)}\varphi(x)]$ and $n$ get $p^\prime$  such that
 \[\lc{p^\prime}[\lc{t(n)}\varphi(n)].\] 
If $t(n)\!:\!\varphi(n)$ does not hold, by completeness of \pa\ with respect to primitive recursive conditions, we would have  $\pa\proves \neg t(n)\!:\!\varphi(n)$, i.e., $\pa\proves\bot$, which was shown earlier to be impossible. Therefore, $t(n)$ is a \pa-derivation of $\varphi(n)$. 
\end{proof}
\begin{Coro} Informal arithmetic proves that proofs of schemes in \pa\ are consistent, i.e., \pa\ does not prove schemes containing $\bot$.
\end{Coro}
\begin{Coro} \textit{Proving a sentence is a special case of proving a scheme}.  
\end{Coro}
Proving schemes does not add any new theorems but rather presents provable formulas in groups, that is, as \textit{schemes}. Formulas from a scheme can be ``concurrently" certified by one finite proof of a scheme.

\begin{Prop}\label{p3}
$\ $ 

a) ``Weakly provable" {\em does not yield} ``Provable."

b) ``Provable" {\em does not yield} ``Strongly provable."
\end{Prop} 
\begin{proof}
a)  Consider the \pa-consistency scheme 
\[ \{\neg\lc{x}\bot\}.\] 
As we have shown, this scheme is provable in \pa. 
By G\"odel's Second Incompleteness Theorem, $\pa\not\proves\forall x\ \neg\lc{x}\bot$, hence the scheme $\{\neg\lc{x}\bot\}$ is not strongly provable in \pa. 

b) The scheme $$\{\neg\lc{x}\Box\bot\}$$ is weakly provable since  $\neg\lc{n}\Box\bot$
is true for each $n=0,1,\ldots$, and hence provable in \pa\ as a true primitive recursive sentence. 
Suppose the scheme $\{\neg\lc{x}\Box\bot\}$ is provable. Then for some selector term $t$, 
\[ \pa\proves\forall x [\lc{t(x)}\neg\lc{x}\Box\bot]. \]
By easy \pa-reasoning, we get 
\[ \pa\proves\forall x\Box\neg\lc{x}\Box\bot,\] 
which is impossible by the aforementioned Kurahashi-Sinclaire observation. 

The scheme $\{\neg\lc{x}\Box\bot\}$ is an example of a property, we suggest calling it \textit{Unprovability of Inconsistency}, which is not provable in \pa, but each of its instances is provable in \pa.
\end{proof}

Neither strong nor weak provability of schemes coheres with Hilbert's consistency program. We argue that provability of schemes is a better fit here. 
Conceptually, proofs of schemes represent an overlooked class of valid reasoning in informal arithmetic in which the induction instances involved are not bounded. This reasoning is {\it de facto} used when proving \textit{Complete Induction} or similar properties, but has been lacking an appropriate formalization within the \pa\ framework. We are now filling this gap with the general theory of proving schemes in \pa. 

It is easy to observe that the following three properties hold of proofs of schemes. First, proofs of schemes are finite syntactic objects. Second, the proof predicate
\[ \mbox{\it ``$\langle t, p\rangle$ is a proof of scheme $\{\varphi\}$"}
\]  
is decidable. Third, the set of provable schemes is recursively enumerable.


There are additional subtleties in proving consistency presented as a scheme. Consider a primitive recursive function $v$ which given $n$ returns a \pa-proof of $\neg\lc{n}\bot$:
\begin{quote}
{\it
Given $x$, check whether $x$ is a proof of $\bot$ in \pa. If ``yes," then put $v(x)$ to be a simple derivation of $\neg\lc{x}\bot$ from $\bot$. If ``no," then use provable $\Sigma_1$-completeness and put $v(x)$ to be a constructible derivation of $\neg\lc{x}\bot$ in \pa. 
}
\end{quote}
Let $p$ be an obvious \pa-proof of $\forall x[\lc{v(x)}\neg\lc{x}\bot]$. Consider two questions: 
\begin{enumerate}
\item Is $\langle v, p\rangle$ a proof of the scheme $\{\neg\lc{x}\bot\}$?
\item Is $\langle v, p\rangle$ a proof of \pa-consistency? 
\end{enumerate}

The answer to (1) is obviously ``yes" since $\langle v, p\rangle$ fits the definition of a proof of the scheme $\{\neg\lc{x}\bot\}$ in \pa.

The answer to (2) is ``no." This question should be understood as whether $v(n)$, as a contentual argument, proves that a \pa-derivation $n$ does not contain $\bot$. The answer to this question is obviously negative: $v(n)$ only tells us that if $n$ contains $\bot$, we would still be able to offer a fake proof of $\neg\lc{n}\bot$. This is not a consistency proof. 
\medskip\par
Proof theorists should not feel disappointed to not see a clean, formal criterion of what counts as a consistency proof. After all, logicians have not had a clean formal criterion of what counts as a \textit{consistency formula}. Researchers have used contentual judgements to rule out suspect formulas, such as the Rosser consistency formula, for years.


\section{Discussion}


There is a long history of suggestions of how to bypass G\"odel's Second  Incompleteness Theorem (cf. \cite{Det79,Det86,Det01,Zach16}). Our approach to proving consistency appears to be novel. It vindicates, to some extent, Hilbert's program of establishing the consistency of formal theories. Thinking of proving consistency of a theory by means formalizable in the same theory should no longer be taboo. As we have seen, our proof of \pa-consistency in informal arithmetic is clearly formalizable in \pa.


By no means are we casting doubt upon G\"odel's Incompleteness Theorems, ordinal analysis, \textit{etc.}; these are the classics of mathematical logic. 
However, representing consistency with the arithmetical formula ${\sf Con}_{\sf PA}$ has distorted the original notion of consistency.
It appears the unprovability of ${\sf Con}_{\sf PA}$ is caused by a technicality, namely the internalized universal quantifier, rather than by deeper foundational problems. This effect is visible in the proof-theoretical format: a provability of ${\sf Con}_{\sf PA}$ would mean that for some proof $p$, 
\[ p:\forall x(\neg\lc{x}\bot),\]
i.e., $p$ should be a proof of $\neg\lc{x}\bot$ for both standard and nonstandard $x$'s, which is impossible since, in some models, $c:\!\bot$ holds for some nonstandard $c$.
Once consistency is considered in its original combinatorial form without unnecessary \textit{a priori} internalization, the consistency property can be viewed as represented by a scheme $$\{\neg\lc{x}\bot\},$$ and its proof in \pa\ is 
\[ p:\forall x[t(x)\!:\!\neg\lc{x}\bot].\]
Since the selector $t(x)$ is a primitive recursive function, it is total, and for each standard $n$ returns a standard proof $t(n)$ of $\neg\lc{n}\bot$, which is sufficient to claim consistency. The abstract idea of using selectors for proving schemes can be traced back to so-called \textit{Kreisel's second clause} in the BHK provability semantics for intuitionistic logic, formalized in \cite{Art01a}, cf. \cite{AF19}.

It is instructive to compare our consistency proof with the well-known infinitary $\omega$-rule 
\[ \ \ \ \ \ \ \ \ \ \ \varphi(0),\ \varphi(1),\ \varphi(2),\ldots\ /\ \forall x\varphi(x). \]
The $\omega$-rule is too strong for our purposes: it steps outside \pa, whereas we stay inside \pa. In fact, we prove the entire collection of premises of $\omega$-rule
\[ \neg\lc{0}\bot, \neg\lc{1}\bot,\neg\lc{2}\bot,\ldots  \]
in a finite way, but not the conclusion $\forall x\ \neg\lc{x}\bot$.

Is informal arithmetic stronger than \pa?
For \pa, we are using Hilbert's notion of proof, i.e., finite derivations in (finite fragments of) \pa. This is probably sufficient for deriving in \pa\ {\em sentences} provable in informal arithmetic. However, \pa-derivations are too weak to express arithmetical reasoning that refers to regular but unbounded collections of induction instances, cf. the proofs in informal arithmetic of \textit{Complete Induction} and \textit{Consistency}. 
Overall, as hinted by G\"odel, informal arithmetic enjoys more freedom for proving than formal \pa. \pa\ with the proofs of schemes capability is a better approximation to informal arithmetic, but this is hardly the end of the story. In summary, we have no objections to the  \textit{Formalization Principle} applied to proofs of sentences. However, the \textit{${\sf Con}_{\sf PA}$ as Consistency Principle} has been refuted. 

Our findings free foundations of verification from some ``impossibility" limitations. Imagine that we want to verify the property 
\begin{equation}\label{ver}
\forall x [t(x)=0]
\end{equation}
for some total computable term $t(x)$ by proving (\ref{ver}) in \pa. 
Given G\"odel's Second Incompleteness Theorem, in addition to a formal proof of (\ref{ver}) in \pa, one needs some consistency assumptions about \pa\ to conclude that $t(n)$ returns $0$ for each $n=0,1,2,\ldots$. Since it was assumed that these additional assumptions could not be verified in \pa, this, strictly speaking, left a certain foundational loophole.
In our framework, \pa\ proves its consistency, and hence these additional meta-assumptions can be dropped. Proving (\ref{ver}) formally in \pa, as we have seen, is certified as a self-sufficient verification method.

\section{Thanks}

Thanks to
\medskip\par
Lev Beklemishev, Sam Buss, John Connor, Michael Detlefsen, Hartry Field, Mel Fitting, Richard Heck, Carl Hewitt, Harold Hodes, John H. Hubbard, Dongwoo Kim, Karen Kletter, Vladimir Krupski, Taishi Kurahashi, Hirohiko Kushida, Roman Kuznets, Yuri Matiyasevich, Richard Mendelsohn, Eoin Moore, Andrei Morozov, Anil Nerode, Elena Nogina, Vladimir Orevkov, Fedor Pakhomov, Rohit Parikh, Vincent Peluce, Brian Porter,  Graham Priest, Michael Rathjen, Andrei Rodin, Chris Scambler, Luke Serafin, Sasha Shen, Richard Shore, Morgan Sinclaire, Stanislav Speransky, Thomas Studer, Albert Visser, Dan Willard, Noson Yanofsky, and many others.

\end{document}